\documentclass[EJP,preprint]{ejpecp} 
\usepackage{xcolor}
\usepackage[utf8]{inputenc}
\usepackage{natbib}
\usepackage[format=plain, font=it, labelfont=it]{caption}
\usepackage{enumitem}
\usepackage[normalem]{ulem}

\SHORTTITLE{Markov processes forced on a subspace by a large drift}

\TITLE{Markov processes forced on a subspace by a large drift, with applications to population genetics}

\AUTHORS{Samuel Ayomide Adeosun\footnote{University of Freiburg, Germany.  \EMAIL{samuel.adeosun@stochastik.uni-freiburg.de}} \and Peter Pfaffelhuber\footnote{University of Freiburg, Germany. \EMAIL{p.p@stochastik.uni-freiburg.de}}}

\AMSSUBJ{92D15} 
\AMSSUBJSECONDARY{60J80; 60F17; 60G57}
\SUBMITTED{February 20, 2026} 
\ACCEPTED{???} 

\VOLUME{0}
\YEAR{2026}
\PAPERNUM{0}
\DOI{}

\ABSTRACT{Consider a sequence of Markov processes $X^1, X^2,...$ with state space $E$, where $X^N$ has a strong drift to $D \subseteq E$, such that $\Phi(X^N)$ is slow for some appropriate $\Phi: E\to D$. Using the method of martingale  problems, we give a limit result, such that $\Phi(X^N) \xRightarrow{N\to\infty} Z$ in the space of càdlàg paths, and $X^N \xRightarrow{N\to\infty} X$ in measure. \\
We apply the general limit result to models for copy number variation of genetic elements in a diploid Moran model of size $N$. The population by time $t$ is described by $X^N \in \mathcal P(\mathbb N_0)$, where $X^N_k$ is the frequency of individuals with copy number $k$, and $\Phi: \mathcal P(\mathbb N_0) \to \mathbb R$ is the first moment.
}

\makeatletter
\renewcommand{\@fnsymbol}[1]{\@arabic{#1 }}
\makeatother

\begin{document}
\section{Introduction}
Slow-fast systems arise frequently in probabilistic models (see e.g.\ \citealp{ball2006asymptotic, berglund2006noise, li2022slow, kifer2024strong, champagnat2025convergence}). We study the situation of a fast evolving sequence of Markov processes $X^N$, such that (i) $Z^N := \Phi(X^N)$ evolves slowly and (ii) $X^N$ is pushed fast towards a slow subset of the state space. A similar situation was studied by \cite{Katzenberger1991} using semi-martingale techniques. However, we do not show convergence of $X^N$ in path space using Lyapunov functions, but rather use tightness and martingale techniques in order to show convergence of $X^N$ in measure, and of $\Phi(X^N)$ in path space. Actually, this approach has appeared in a special situation in \cite{pfaffelhuber2023diploid}, but is here carried out in full generality. 

As our main application, we use a population genetic model, where reproduction involves two parents and each individual has a type in $\mathbb N_0$, counting the number of genetic elements it carries. The model is fully specified once we specify the distribution of genetic elements a parent gives to its offspring. Such models build on the two-parental Moran model and have e.g.\ been studied in \cite{coron2022pedigree, otto2022recombination, otto2023structured, pfaffelhuber2023diploid, omole2025population}.

\section{The abstract result}
Recall that a process $X = (X_t)_{t\geq 0}$ with complete and separable metric state space $(E,r)$ solves the $(G, \mathcal D)$-martingale problem for some linear $G : \mathcal D \subseteq \mathcal C_b(E) \to \mathcal C_b(E)$ (where $\mathcal C_b(E)$ is the set of real-valued, bounded continuous functions on $E$) if 
$$ \Big( f(X_t) - \int_0^t Gf(X_s)ds \Big)_{t\geq 0}$$
is a martingale for all $f \in \mathcal D$. A Markov process is the unique solution to its martingale problem (when taking $\mathcal D$ large enough), and if there is a unique such solution, it is a strong Markov process (see e.g.\ Theorems 4.3.1 and 4.3.2 in \cite{EthierKurtz1986}). Usually, such a process has càdlàg paths $[0, \infty) \to E$ and we denote the set of such paths by  $\mathcal D(E)$; see Theorem 4.3.6 in \cite{EthierKurtz1986}.

Assume we have a sequence of Markov processes $X^1, X^2, ...$ with state space $(E,r_E)$ such that the generator $G^N$ of $X^N$ has domain $\mathcal D_E \subseteq \mathcal C_b(E)$ and is of the form 
\begin{align}\label{eq:ass0}
  G^N = N G_1^N + G_0^N.
\end{align}
We are interested in the weak limit of $X^N$ as $N\to\infty$, in the special situation that $D$ is another Polish space and $\Phi : E \to D$ is such that, for some $\mathcal D_D \subseteq \mathcal C_b(D)$, if $g \in \mathcal D_D$, we have $g \circ \Phi \in \mathcal D_E$ and
\begin{align}
    \label{eq:ass2}
    G_1 (g \circ \Phi) = 0.
\end{align}
(In other words, the dynamics given by $G_1^N$ change $X^N$ fast, but does not change $\Phi(X^N)$.) Recall that there are two kinds of convergence on $\mathcal D(E)$. First, the usual {\it Skorohod convergence}; see e.g.\ Chapter~3 in \cite{EthierKurtz1986}. Second, there is {\it convergence in measure}: Define the  {\it weighted occupation measure} of $\xi \in \mathcal D(E)$, as the  probability measure
\begin{align}\label{occmeas}
    \Gamma_\xi([0,t] \times A) := \int_0^t e^{-s} \mathbf 1_{\{\xi_s\in A\}} ds,
\end{align}
where $t \ge 0$ and $A$ is a measurable subset of $E$. Following \cite{Kurtz1991} we say that a sequence $(\xi^N)$ in $\mathcal D(E)$ {\it converges  in measure} to $\xi \in \mathcal D(E)$ if the sequence of probability measures  $\Gamma_{\xi^N}$ converges weakly to $\Gamma_{\xi}$. 
We assume that
\begin{enumerate}
\item[A1] $(\Phi(X^N), X^N) \xRightarrow{N\to\infty} (Z, X)$ for some $(Z,X)$, where convergence to $Z$ is with respect to the Skorohod topology in $\mathcal D(E)$, and to $X$ in measure. 
\end{enumerate}
For A1 to hold along a subsequence, it suffices to assume that $\Phi(X^N)_{N=1,2,...}$ is tight (in the space of càdlàg paths on $D$), and $(X^N)_{N=1,2,...}$ are tight in measure, i.e.\ the sequence of occupation measures is tight. 
\begin{enumerate}
\item[A2] There is $G_0$ such that, for all $f \in \mathcal D_E$, 
\begin{align*}
G_0^N f(X^N) & \xRightarrow{N\to\infty} G_0 f(X)
\end{align*}
in measure.
\item[A3] There is $\Xi: D \to E$ and $\mathcal D_E' \subseteq \mathcal D_E$ with the following property: \\ If $G_1f(x) = 0$ for all $f \in \mathcal D_E'$ and $x\in E$, then $x = \Xi(\Phi(x))$. (In other words, we can recover $x$ if we are given $\Phi(x)$ and $G_1f(x) = 0$.)
\end{enumerate}

Then, we have the following
\begin{theorem}
    \label{T1}
    Let $(E,r_E)$ and $(D, r_D)$ be complete and separable spaces, $\mathcal D_E \subseteq \mathcal C_b(E)$ and $\mathcal D_D \subseteq \mathcal C_b(D)$, as well as $\Phi: E\to D$ such that $g \circ \Phi \in \mathcal D_E$ for $g \in \mathcal D_D$. Assume $X^N$ is solution of the $(G^N, \mathcal D_E)$-martingale problem with $G^N$ as in \eqref{eq:ass0} and $G_1^N$ satisfies \eqref{eq:ass2}. If A1, A2 and A3 hold, $\Phi(X^N) \xRightarrow{N\to\infty} Z$, where $Z$ solves the martingale problem for $g \mapsto (G_0(g\circ \Phi)) \circ \Xi$ with $g \in \mathcal D_D$.
\end{theorem}

\begin{proof}
From A1, assume that $\Phi(X^N) \xRightarrow{N\to\infty} \Phi(X)$ weakly in the space of càdlàg paths, and $X^N \xRightarrow{N\to\infty} X$ weakly in measure. Then, for $ f \in \mathcal D_E'$, using \eqref{eq:ass0},
\begin{align*}
\Big(\tfrac 1N f(X^N_t) & - \int_0^t \big(G_1 f + \tfrac 1N G_0^N f\big)(X_s^N) ds\Big)_{t\geq 0} \xRightarrow{N\to\infty} \Big(\int_0^t G_1 f(X_s)ds\Big)_{t\geq 0}    
\end{align*}
is a martingale, so the right hand side is a martingale with bounded variation, hence vanishes, i.e.\ $G_1f(X_s) = 0$ for Lebesgue almost all $t$. From A3, this implies that $X_t = \Xi(\Phi(X_t))$ for Lebesgue almost all $t\ge 0$. 

In order to put everything together, consider $g \in \mathcal D_D$ and note that $g \circ \Phi \in \mathcal D_E$, hence using \eqref{eq:ass2} and $X^N \xRightarrow{N\to\infty} \Xi(\Phi(X))$ in measure,
\begin{align*}
\Big(g (\Phi(X^N_t)) - & \int_0^t G_0^N (g\circ \Phi)(X_s^N) ds\Big)_{t\geq 0} \notag\\ & \xRightarrow{N\to\infty} 
\Big(g (\Phi(X_t)) - \int_0^t G_0 (g\circ \Phi)(\Xi(\Phi(X_s))) ds\Big)_{t\geq 0}    
\end{align*}
is a martingale. In particular, $\Phi(X)$ solves the martingale problem for $g \mapsto (G_0(g\circ \Phi)) \circ \Xi$ with $g \in \mathcal D_D$.
\end{proof}

~~

\noindent
Let us consider the following simple example: Let $E = \mathbb R^2$,
$$ G_1f(x,y) = (x - y) (\partial_y f(x,y) - \partial_x f(x,y)), \qquad G_0 = \tfrac 12 (\partial_{xx} f + \partial f_{yy})(x,y).$$
So $X^N$ is some Brownian motion in $\mathbb R^2$ with a strong force to the diagonal. Taking $\Phi(x,y) = \tfrac 12(x + y)$ and $\mathcal D_E = \mathcal D'_E:= \mathcal C^2(\mathbb R \times \mathbb R)$, we find that $G_1f(x) = 0$ for all $f \in\mathcal D_E$ implies that $x$ is on the diagonal, i.e. $x = \Xi(\Phi(x))$ with $\Xi(z) := (z,z)$. So, we have that the limit $\Phi(X)$, has generator
$$ Hg(x) = G_0 (g \circ \Phi)(\Xi(z)) = \tfrac 12\big( \partial_{xx} + \partial_{yy}\big)(g \circ \Phi)(\Xi(z)) = \tfrac 12 g''(z).$$
So, as anticipated, the limit of $\tfrac 12(X^N + Y^N)$ is a Brownian motion. In addition, $X^N - Y^N \xRightarrow{N\to\infty} 0$ in measure.

~


\section{Modeling copy number variation of genetic elements}
\label{S:model}
\noindent
Here is an extension of the population model for diploid organisms from \cite{pfaffelhuber2023diploid}, which we will study in detail in the remainder of the paper:
\begin{itemize}
    \item A diploid population of constant size $N$ consists of individuals, each carrying a certain number of genetic elements.
    \item Reproduction events occur at rate $N^2$. Upon a reproduction event, choose individuals $a,b,c$. Individual $c$ dies, and is replaced by offspring from $a$ and $b$.
    \item For probability distributions $p_k^N, k=0,1,2,...$, if individual $a$ has $k$ genetic elements, it transfers a random number of genetic elements to its offspring, distributed like $p_k^N$. Individual $b$ inherits an independent number of genetic elements, distributed as $p_l^N$, if $b$ has $l$ genetic elements. We call $(p_k^N)_{k=0,1,2,...}$ the family of inheritance distributions and we assume that there is $(p_k)_{k=0,1,2,...}$ with 
    \begin{equation}
        \label{eq:pk}
        \begin{aligned}
        (a) & \; \sum_{k} k p_k = \tfrac k2, \\ (b) & \; N(p_k^N(.) - p_k(.)) \xrightarrow{N\to\infty} r_k(.) \text{ with } \sum_j j r_k(j) = \alpha k, k \in \mathbb N
        \end{aligned}    
    \end{equation}
    for some $r_k$ and some $\alpha \in\mathbb R$. Note that 
    \begin{align}
        \label{eq:rk}
        \sum_j r_k(j) = N \sum_j p_k^N(j) - p_k(j) = 0.
    \end{align}
\end{itemize}
Recall that $p^N_k := p_k := B(k,1/2)$ was studied in \cite{pfaffelhuber2023diploid}. We want to study the evolution of the distribution of these genetic elements in the limit $N\to\infty$. Therefore, let $\mathcal P(\mathbb N)$ be the set probability distributions on $\mathbb N _= \{0,1,2,...\}$, equipped with the topology of weak convergence, and $X^N$ be the $\mathcal P(\mathbb N)$-valued Markov jump process describing the evolution of copy numbers in the population with $N$ individuals. For $p_k^N$, we study the two cases: 
\begin{align} \label{eq:twocases}
\text{(i)} \;  p_k^N = B(k, \tfrac 12 + \varepsilon_N) \text{ with } N\varepsilon_N \xrightarrow{N\to\infty} \alpha, \qquad \text{(ii)} \; p_k^N = p_k = U(\{0,...,k\})
\end{align}
and show that \eqref{eq:pk} holds at the beginning of Sections~\ref{ss:binom} and \ref{ss:uniform}.

\begin{remark}[Motivation for (i) and (ii)]
The case (i) extends earlier work of \cite{pfaffelhuber2023diploid} to a case with bias. For (ii), we refer to \cite{otto2022recombination} and \cite{otto2023structured}. The idea behind this model is that there is a uniformly distributed breakpoint within the $k$ and $l$ genetic elements from both parents, and both parents inherit only the part on one side of the breakpoint to the offspring. Clearly, in this case, \eqref{eq:pk} holds with $r_k = \alpha = 0$.    \qed
\end{remark}

By the dynamics from above, $X^N$ jumps from $x$ to $x + e_m - e_n$ (where $e_m$ is the $m$th unit vector) at rate
\begin{align}
    \label{eq:jumpRate}
    \lambda^N_{m,n}(x) := \tfrac {N^2}2 x_n & \sum_{k,l} x_k x_l \sum_j p^N_{k}(j)p^N_l(m-j).
\end{align}
In our models, a parent carrying $k$ genetic elements, inherits on average (i) $(\tfrac 12 + \varepsilon_{N})k$ ((ii) $\tfrac k2$) genetic elements. As we will see now, the mean number 
\begin{align}
    \label{eq:mean}
    \Phi(X^N) := \rho_1(X^N) := \sum_{j=0}^\infty j X^N_j
\end{align}
evolves slowly, while $X^N$ is fast. 

\begin{theorem}\label{T2}
    Let $X^N$ be a Markov jump process with state space $\mathcal P(\mathbb N)$ and transition rates given by $\lambda_{m,n}^N$ from \eqref{eq:jumpRate}.   Assume that, for some $z > 0$, $\Phi(X^N_0) \xrightarrow{N\to\infty} z$ in probability, and $\sup_N  \mathbf E\Big[\sum_{k=1}^\infty k^3 X^N_0(k)\Big] < \infty$.     
    \begin{enumerate}
    \item If $p_n^N$ satisfies \eqref{eq:twocases}(i), $(X^N, \Phi(X^N)) \xRightarrow{N\to\infty} (X, Z)$, where $Z$ solves
    $$ dZ = \alpha Z dt + \sqrt{Z}dW,$$
    and $X_t = \text{Poi}(Z_t)$ for all $t\ge 0$. 
    \item If $p_n^N$ satisfies \eqref{eq:twocases}(ii), $(X^N, \Phi(X^N)) \xRightarrow{N\to\infty} (X, Z)$, where $Z$ solves
    $$ dZ = \sqrt{Z(Z+2)}dW,$$
    and $X_t = \text{NB}(2,2 / (2 + Z_t)$ (the negative binomial distribution with number of successes 2 and mean $Z_t$.)
    \end{enumerate}
    In both cases, $\Phi(X^N) \xRightarrow{N\to\infty} Z$ in path space, and $X^N \xRightarrow{N\to\infty}X$ in measure.  
\end{theorem}

\begin{remark}[More general result for general $p_k$]
As the structure of the Theorem suggests, there is a correspondence between the family $(p_k^N)_{k=0,1,...}$, the form of $X_t$ given $Z_t$, which holds for all $t>0$, and the dynamics of $Z$. For the former, there is for each choice of $(p_k^N)_{k=0,1,...}$ a family of distributions $(q_z)_{z \geq 0}$ for the limit $X$, which is parameterized by its mean, $Z$, i.e.\ $X_t = q_{Z_t}$ for all $t$. This connection will be made in Lemma~\ref{l:4} for (i) and Lemma~\ref{l:5} for (ii). As for the dynamics of $Z$, note that the diffusion term is governed by variance of $q_z$: In case (i), we have $q_z = \text{Poi}(z)$ with variance $z$. In case (ii), we have $q_z = \text{NB}(2, 2/(2+z))$, which has variance $z (z + 2)$. This connection holds in greater generality as we will see in Lemma~\ref{l42}.
\end{remark}

\begin{remark}[$p_n = \tfrac 12 \delta_n + \tfrac 12 \delta_n$]\label{rem:noconv}
Yet another canonical choice for the inheritance distributions is $p_n = \tfrac 12 \delta_n + \tfrac 12 \delta_n$ for $n=0,1,2,...$, i.e.\ a parent either inherits no or all of its genetic elements to the offspring, each with probability $\tfrac 12$. This case can be studied directly, since $Y := X(0)$ is an autonomous process. We have that $Y$ jumps from $y$ to
\begin{align*}
    y + 1 & \text{ at rate } N^2 (1-y) \big( \tfrac 14 (1-y)^2 + (1-y)y + y^2 \big), \\
    y - 1 & \text{ at rate } N^2 y\big( \tfrac 34 (1-y)^2 + (1-y)y \big).
\end{align*}
For example, the term $\tfrac 14 N^2 (1-y)^3$ takes into account all events where all individuals involved in the reproduction event have at least $1$ genetic element, and both parents choose to inherit none of their genetic elements. Since 
\begin{align*}
    (1-y) \big( & \tfrac 14 (1-y)^2 + (1-y)y + y^2 \big) - y\big( \tfrac 34 (1-y)^2 + (1-y)y \big) \\ & \geq     (1-y) \big( (1-y)y + y^2 \big) - y\big( (1-y)^2 + (1-y)y \big) = 0,
\end{align*}
this shows that $N-Y$ is a non-negative supermartingale, hence converges to $0$ almost surely, in the time-scale $N^2 dt$. In other words, although $p_N$ has mean $\tfrac n2$, we find that in the limit $N\to\infty$, no individual carries any genetic element.

\end{remark}

\section{Preparation}
\label{S:preparation}
The proof of Theorem~\ref{T2} is based on an application of Theorem~\ref{T1}. In this section, we will prepare the proof for some of the assumptions of Theorem~\ref{T1} for the population model from above. We will as long as possible keep a general  family of inheritance distributions $(p_k^N)_{k=0,1,...}$. This means that we only assume a certain form for first three factorial moments of $p_k$ and $r_k$; see \eqref{eq:an} and \eqref{eq:bn}. In Section~\ref{S:later}, we restrict ourselves to the cases (i) and (ii) and finalize the proof of Theorem~\ref{T2} in both cases.  

For the generator $G^N$ of $X^N$ and $f \in \mathcal D := \mathcal C^2_b(\mathcal P(\mathbb N))$, we see directly from the jump rates \eqref{eq:jumpRate}, using \eqref{eq:pk} and \eqref{eq:rk},
\begin{equation}\label{eq:GN}
    \begin{aligned}
    G^N f(x) & = \frac{N^2}{2} \sum_{n}  x_n \sum_{k,l} x_k x_l    \sum_{j,m} p_k^N(j) p_l^N(m-j) \big(f(x + (e_m - e_n)/N) - f(x)\big) \\ & = (N G_1 + G_0) f(x) + o(1) \text{ with} \\
    G_1 f(x) & = \frac 1 2 \sum_{n}  x_n \sum_{k,l} x_k x_l \sum_{j,m} p_k(j) p_l(m-j) (e_m - e_n) \cdot \nabla f(x) \\ 
    G_0 f(x) & = 
    \frac 1 2 \sum_{k,l} x_k x_l \sum_{j,m} (p_k(j) r_l(m-j) + r_k(j) p_l(m-j))e_m \cdot \nabla f(x) \\ & \qquad +     
    \frac{1}{4} \sum_{n}  x_n \sum_{k,l} x_k x_l \sum_{j,m} p_k(j) p_l(m-j) (e_m - e_n) \cdot \nabla^2 f(x) \cdot (e_m - e_n).
\end{aligned}    
\end{equation}
This shows that the form \eqref{eq:ass0} applies, and A2 holds (with $G_0^N = G_0 + o(1)$), provided that $X^N \xRightarrow{N\to\infty}X$ in measure. Moreover, for \eqref{eq:ass2}, using \eqref{eq:pk}, recalling $\Phi$ from \eqref{eq:mean}, with $g \in \mathcal C_b^1(\mathbb R_+)$,
\begin{align*}
    G_1 (g \circ \Phi) & = \frac 12 g'(\Phi(x)) \sum_{n}  x_n \sum_{k,l} x_k x_l \sum_{j,m} p_k(j) p_l(m-j) (j + (m - j) - n) \\ & = \frac 12 g'(\Phi(x)) \Big( 2\sum_k x_k \sum_j j p_k(j) - \rho_1(x)\Big) = 0.
\end{align*}
For the remaining tasks, A1 and A3, we will be using for $x \in \mathcal P(\mathbb N)$ the generating function $s\mapsto \psi_s(x)$ and the factorial moments, $\rho_k(x), k=1,2,...$,
\begin{align}
    \label{eq:taylor}
        \psi_s(x) & = \sum_{n=0}^\infty x_n (1 - s)^n = \sum_{k=0}^\infty \rho_k(x) \frac{(-s)^k}{k!} = 1 - s \rho_1(x) + \tfrac 12 s^2 \rho_2(x) - \tfrac 16 s^3 \rho_3(x) + O(s^4) \intertext{with}
        \notag
        \rho_n(x) & = (-1)^n \frac{\partial^n}{\partial s^n} \psi_s(x) \bigg|_{s=0} = \sum_k k \cdots (k-n+1) x_k, \quad k=0,1,2,...
\end{align}
Let us consider the dynamics on the fast time-scale, i.e.\ let us look at A3. We set 
\begin{align*}
\mathcal D' := \text{algebra generated by }\{\psi_s : \mathcal P(\mathbb N) \to \mathbb R_+ : s \in [0,1]\}
\end{align*}
and now consider the corresponding dynamics:

\begin{lemma}[Dynamics on the fast time-scale]
\label{l41}
It holds
\begin{align}\label{eq:Gqpsishow}
    2G_1\psi_s(x) & = \Big( \sum_{k} x_k \psi_s(p_k)\Big)^2 - \psi_s(x). 
\end{align}
In addition, if $p_k$ is such that \eqref{eq:pk} holds and for suitable $a_2, a_3$,
\begin{align}
    \label{eq:an}
    \rho_2(p_k) = a_2 k(k-1), \qquad \rho_3(p_k) = a_3 k(k-1)(k-2).
\end{align}
Then, 
\begin{equation}\label{eq:dynrho23}
\begin{aligned}
    2G_1 \rho_2(x) & = \tfrac 12 \rho_1^2(x) - (1 - 2a_2) \rho_2(x), \\
    2G_1 \rho_3(x) & = 3a_2 \rho_2(x) \rho_1(x) - (1 - 2a_3) \rho_3(x).
\end{aligned}    
\end{equation}
\end{lemma}

\begin{proof}
For the first assertion, we write using \eqref{eq:GN}
\begin{align*}
    2G_1\psi_s(x) & = \sum_{n} x_n \sum_{k,l} x_k x_l \sum_{j,m} p_k(j) p_l(m-j) \big((1-s)^{j + (m-j)} - (1 - s)^n\big) \\ & = \sum_{k,l} x_k x_l \psi_s(p_k)\psi_s(p_l) - \sum_{n} x_n (1 - s)^n = \Big( \sum_{k} x_k \psi_s(p_k)\Big)^2 - \psi_s(x).
\end{align*}
Recalling that $\rho_1(p_k) = \tfrac 12 k$ by assumption (see \eqref{eq:pk}), use \eqref{eq:taylor} and \eqref{eq:an} in order to write
\begin{equation}
    \label{eq:411}
    \begin{aligned}
        \sum_{k} & x_k \psi_s(p_k) = \sum_{k} x_k (1 - s \rho_1(p_k) + \tfrac 12 s^2 \rho_2(p_k) - \tfrac 16 s^3 \rho_3(p_k) + O(s^4)) \\ & = 1 - \tfrac 12 s \rho_1(x) + \tfrac 12 s^2 a_2 \rho_2(x) - \tfrac 16 a_3 s^3 \rho_3(x) + O(s^4),
    \end{aligned}
\end{equation}
which implies (writing $\rho_i := \rho_i(x), i=1,2,3$)
\begin{equation}\label{eq:square}
\begin{aligned}
    \Big(\sum_{k} x_k \psi_s(p_k)\Big)^2 & = 1 - s \rho_1(x) + \tfrac 12 s^2\big(\tfrac 12 \rho_1^2 + 2 a_2 \rho_2\big) - \tfrac 16 s^3\big(3 a_2 \rho_2 \rho_1 + 2 a_3 \rho_3 \big) + O(s^4).
\end{aligned}    
\end{equation}
Therefore, evaluating \eqref{eq:Gqpsishow} as a series in $s$, and comparing coefficients, we obtain \eqref{eq:dynrho23}.
\end{proof}

With Lemma~\ref{l41}, we will show in Sections~\ref{ss:binom} for the binomial case (Lemma~\ref{l:4} and Remark~\ref{r:4}) and~\ref{ss:uniform} for the uniform case (Lemma~\ref{l:5} and Remark~\ref{r:5}) that 
\begin{equation}\label{eq:astast}
\begin{minipage}{0.8\textwidth}
    $G_1 \psi_s(x) = 0$ iff $x$ is Poisson (negative binomial) whenever $(p_k)_{k=0,1,2,...}$ is binomial (uniform).\end{minipage}
\end{equation}
In other words, if $G_1 \psi_s(x)(x)=0$ for all $s \geq 0$, we can define $\Xi(z) = \text{Poi}(z)$ in case (i) and $\Xi(z) = \text{NB}(2, 2 / (2 + z))$ in case (ii) and have $x = \Xi(\Phi(x))$, as needed for A3. 

For A1, i.e.\ tightness, we need to consider the slow time-scale as well. The corresponding calculations will be carried out in Lemma~\ref{l42}, Lemma~\ref{l43}, and Lemma~\ref{l44}, which lead to a proof of tightness in Proposition~\ref{p:tight}. 

\begin{lemma}[Dynamics of $\Phi(X^N)$ on the slow time-scale and limiting generator]
\label{l42}
Assume $r_k$ is as in \eqref{eq:pk}, as well as (see Lemma~\ref{l41}) $\rho_2(p_k) = a_2 k(k-1)$ for some $a_2$. Then, 
\begin{align}\label{eq:astastG0}
    G_0(g \circ \Phi)(x) & = g'(\Phi(x)) \alpha \Phi(x) + \tfrac 12 g''(\Phi(x)) \Big( (a_2 + \tfrac 12) \rho_2(x) +  \rho_1(x) - \tfrac 34 \rho_1^2(x) \Big) \Big).
\end{align}
In addition, if $x$ solves $G_1 \psi(x) = 0$, i.e.\ $x = \Xi(z)$ with $z = \Phi(x)$,
\begin{align}\label{eq:astastast}
    G_0 (g \circ \Phi) \circ \Xi(z) = \alpha z g'(z) + \tfrac 12 g''(z) v(\Xi(z)),
\end{align} 
where 
$$ v(x) := \rho_2(x) + \rho_1(x) - \rho_1^2(x)$$
is the variance. 
\end{lemma}

\begin{proof}
Using that $\sum_k x_k = 1$ and $\sum_j r_k(j) = 0$ several times, as well as $(m-n)^2 = (j + m-j - n)^2 = j(j-1) + j + (m-j)(m-j-1) + (m-j) + n(n-1) + n + 2j(m-j) - 2jn - 2(m-j)n$,
\begin{align*}
    G_0 & (g\circ \Phi)(x) = \frac{1}{2} \sum_{n}  x_n \sum_{k,l} x_k x_l \sum_{j,m} (p_k(j) r_l(m - j) + r_k(j) p_l(m-j)) \cdot (m - n) g'(\Phi(x)) \\ & \qquad \qquad \qquad + \frac 14 \sum_{n} x_n \sum_{k,l} x_k x_l \sum_{j,m} p_k(j) p_l(m-j) \cdot (m - n)^2 g''(\Phi(x)) \\ & =  \frac 12 g'(\Phi(x)) \sum_{k,l} x_k x_l \sum_j (j r_k(j) + j r_l(j)) \\ & \qquad \qquad \qquad \qquad + \frac 14 g''(\Phi(x)) \Big( \big( \sum_k x_k (2\rho_2(p_k) + 2\rho_1(p_k) - 4\rho_1(p_k) \rho_1(x)\big) 
    \\ &  \qquad \qquad \qquad \qquad \qquad \qquad \qquad \qquad \qquad \qquad + 2 \Big( \sum_k x_k \rho_1(p_k)\Big)^2 + \rho_2(x) + \rho_1(x) \Big) \\ & =
    g'(\Phi(x))\sum_{k} x_k \alpha k
    \\ & \qquad + g''(\Phi(x)) \Big(\tfrac 12 a_2 \rho_2(x) + \tfrac 14 \big(\rho_1(x) - 2 \rho_1^2(x) + \tfrac 12 \rho_1^2 (x) + \rho_2(x) + \rho_1(x) \big) \Big)
    \\ & = g'(\Phi(x)) \alpha \Phi(x) + \tfrac 12 g''(\Phi(x)) \Big( a_2 \rho_2(x) + \tfrac 12 \rho_2(x) +  \rho_1(x) - \tfrac 34 \rho_1^2(x) \Big) \Big),
\end{align*}
which is the first assertion. Next, recall that $\Xi(z)$ solves $G_1 \psi_s(\Xi(z)) = 0$. Take two derivatives at $s=0$ in $G_1 \psi_s(x) = 0$, use \eqref{eq:Gqpsishow}, and write using \eqref{eq:dynrho23}
\begin{align*}
    \rho_2(\Xi(z)) & = \frac{\partial^2}{\partial s^2} \Big( \sum_k \Xi(z)_k \psi_s(p_k) \Big)^2\Big|_{s=0} = 2\Big( \sum_k \Xi(z)_k \rho_2(p_k) + \Big( \sum_k \Xi(z)_k \rho_1(p_k)\Big)^2 \Big) \\ & = 2(a_2 \rho_2(\Xi(z)) + \tfrac 14 \rho_1^2(\Xi(z))).
\end{align*}
Therefore, we finish the proof with
\begin{align*}
    G_0 (g \circ \Phi) \circ \Xi(z) & = g'(z) \alpha z + \tfrac 12 g''(z) \big( \rho_2(\Xi(z)) + \rho_1(\Xi(z)) - \rho_1^2(\Xi(z))\big) \\ & = \alpha z g'(z) + \tfrac 12 g''(z) v(\Xi(z)).
\end{align*}
\end{proof}

We need some more bounds for the slow time-scale:

\begin{lemma}[Dynamics on the slow time-scale, $G_0$]\label{l43}
    It holds
    \begin{align*}
        G_0 \psi_s(x) & = \Big(\sum_{k} x_k \psi_s(p_k)\Big)\Big(\sum_l x_l \psi_s(r_l)\Big), \\
        G_0 \psi_s(x)\psi_r(x) & = \psi_s(x) G_0\psi_r(x) +  \psi_r(x) G_0\psi_s(x) \\ & \quad + \tfrac 14 \Big(\sum_{k} x_k \psi_{s + r - rs}(p_k)\Big)^2 - \tfrac 14 \psi_s(x) \Big(\sum_{k} x_k \psi_{r}(p_k)\Big)^2 \\ & \qquad  - \tfrac 14 \psi_r(x) \Big(\sum_{k} x_k \psi_{s}(p_k)\Big)^2 + \tfrac 14\psi_{s + r - sr}(x).
    \end{align*}
    In addition, with $a_2, a_3$ from Lemma~\ref{l41}, and if \eqref{eq:pk} holds and for suitable $b_2, b_3$,
    \begin{align}
        \label{eq:bn}
        \rho_1(r_k) = \alpha k, \qquad \rho_2(r_k) = b_2 k(k-1), \qquad \rho_3(r_k) = b_3 k(k-1)(k-2),
    \end{align}
    then
    \begin{equation}\label{eq:412}
    \begin{aligned}
        G_0 \rho_1(x) & = \alpha \rho_1(x), \\ 
        G_0 \rho_1^2(x) & = (2 \alpha - \tfrac 34) \rho_1^2(x) + (a_2 + \tfrac 12)\rho_2(x) + \rho_1(x), \\ 
        G_0 \rho_2(x) & = \alpha \rho_1^2(x) + b_2 \rho_2(x), \\
        G_0 \rho_1^3(x) & = (3 \alpha - \tfrac 94)\rho_1^3(x) + 3(a_2 + \tfrac 12)\rho_2(x)\rho_1(x) + 3\rho_1^2(x), \\ 
        G_0 \rho_2(x)\rho_1(x) & = (a_2 - b_2 +\tfrac 12 (1 - 2\alpha))\rho_2(x) \rho_1(x) \\ & \qquad \qquad \qquad + \tfrac 12(\tfrac 12 + 2\alpha)\rho_1(x)^3 - (2 a_2 + 1)\rho_2(x) - \tfrac 12 \rho_1(x)^2 , \\
        G_0 \rho_3(x )& = (3 \alpha a_2 + \tfrac 32 b_2)\rho_2(x) \rho_1(x) + b_3 \rho_3(x).
    \end{aligned}
    \end{equation}
\end{lemma}

\begin{proof}
We start with a similar calculation as in \eqref{eq:411}, replacing $p_k$ by $r_k$, leading to (note that $\psi_0(r_k) = \sum_j r_k(j) = 0$)
\begin{equation}
    \label{eq:414}
    \begin{aligned}
        \sum_{k} & x_k \psi_s(r_k) = - \alpha s  \rho_1(x) + \tfrac 12 s^2 b_2 \rho_2(x) - \tfrac 16 b_3 s^3 \rho_3(x) + O(s^4). 
    \end{aligned}
\end{equation}
For the first assertion, compute 
\begin{equation}\label{eq:413}
\begin{aligned}
G_0 \psi_s(x) & = \frac 12 \sum_{k,l} x_k x_l \sum_{j,m} (p_k(j) r_l(m-j) + r_k(j)p_l(m-j))(1-s)^{j + (m-j)} \\ & = \frac 12 \sum_{k,l} x_k x_l (\psi_s(p_k) \psi_s(r_l) + \psi_s(r_k) \psi_s(p_l)) = \Big(\sum_{k} x_k \psi_s(p_k)\Big)\Big(\sum_l x_l \psi_s(r_l)\Big) \\
& = s\Big(1 - \tfrac 12 s \rho_1(x) + \tfrac 12 s^2 a_2 \rho_2(x) + O(s^3) \Big) \\ & \qquad \qquad \qquad \qquad \qquad \cdot \Big(- \alpha \rho_1(x) + \tfrac 12 s b_2 \rho_2(x) - \tfrac 16 b_3 s^2 \rho_3(x) + O(s^3)\Big) \\ & = - \alpha s \rho_1(x) + \tfrac 12 s^2 (\alpha \rho_1^2(x) + b_2 \rho_2(x)) \\ & \qquad \qquad \qquad - \tfrac 16 s^3(3 \alpha a_2 \rho_2(x) \rho_1(x) + \tfrac 32 b_2\rho_2(x) \rho_1(x) + b_3 \rho_3(x)) + O(s^4),
\end{aligned}    
\end{equation}
which gives the first, third and sixth term in \eqref{eq:412} by comparison of coefficients.
Note that the first, second, and fourth equalities in \eqref{eq:412} can be read from \eqref{eq:astastG0}. For the fifth equality, we need to compute mixed terms. For these, note that $G_0$ consists of both, a first and a second derivative. So, we can write, using \eqref{eq:taylor}, \eqref{eq:square}, and \eqref{eq:413},
\begin{align*}
    2G_0 \psi_s(x)\psi_t(x) & = 2\psi_t(x) G_0\psi_s(x) + 2 \psi_s(x) G_0\psi_t(x) \\ & + \sum_n x_n \sum_{k,l} x_k x_l \sum_{j,m} p_k(j) p_l(m-j)\big( (1-t)^m - (1-t)^n\big)\big( (1-s)^m - (1-s)^n\big) \\ & = 
    2\psi_t(x) \Big(\sum_{k} x_k \psi_s(p_k)\Big)\Big(\sum_l x_l \psi_s(r_l)\Big) + 2\psi_s(x) \Big(\sum_{k} x_k \psi_t(p_k)\Big)\Big(\sum_l x_l \psi_t(r_l)\Big) \\ & \qquad + 
    \Big(\sum_{k} x_k \psi_{t + s - st}(p_k)\Big)^2 - \psi_t(x) \Big(\sum_{k} x_k \psi_{s}(p_k)\Big)^2 \\ & \qquad \qquad \qquad - \psi_s(x) \Big(\sum_{k} x_k \psi_{t}(p_k)\Big)^2 + \psi_{t + s - st}(x)\\ & = 
    \psi_t(x)\Big(\sum_{k} x_k \psi_s(p_k)\Big)\Big(2\sum_l x_l \psi_s(r_l) - \sum_l x_l \psi_s(p_l)\Big) \\ & \qquad +   \psi_s(x)\Big(\sum_{k} x_k \psi_t(p_k)\Big)\Big(2\sum_l x_l \psi_t(r_l) - \sum_l x_l \psi_t(p_l)\Big) 
    \\ & \qquad \qquad \qquad + \Big(\sum_{k} x_k \psi_{t + s - st}(p_k)\Big)^2 + \psi_{t + s - st}(x),
\end{align*}
which leads to, up to second order in $s$ and first order in $t$,
\begin{align*}
    2G_0 \psi_t(x)\psi_s(x) & = (1 - t\rho_1)(1 - \tfrac 12 s \rho_1 + \tfrac 12 s^2 a_2 \rho_2)(-2 \alpha s \rho_1 + s^2 b_2 \rho_2 - 1 + \tfrac 12 s \rho_1 - \tfrac 12 s^2 a_2 \rho_2) \\ & \quad + (1 - s\rho_1 + \tfrac 12 s^2 \rho_2)(1 - \tfrac 12 t \rho_1)(-2 \alpha t \rho_1 - 1 + \tfrac 12 t \rho_1) \\ & \quad 
    + (2 - 2(t + s - st)\rho_1 + \tfrac 12 (t + s - st)^2(\tfrac 12 \rho_1^2 + (2 a_2 + 1)\rho_2) 
    \\ & = 
    (1 - t\rho_1)(1 - \tfrac 12 s \rho_1 + \tfrac 12 s^2 a_2\rho_2)(-1 + s(\tfrac 12 - 2\alpha) \rho_1 + s^2 (b_2 - \tfrac 12 a_2)\rho_2) \\ & \quad + (1 - s\rho_1 + \tfrac 12 s^2 \rho_2)(1 - \tfrac 12 t \rho_1)(-1 + t( \tfrac 12 -2 \alpha) \rho_1 ) \\ & \quad 
    + (2 - 2(t+s-st)\rho_1 + \tfrac 12 (s^2 + 2st -2s^2t )(\tfrac 12 \rho_1^2 + (2 a_2 + 1)\rho_2))
    \\ & = 
    (1 - t\rho_1)(- 1 + (1 - \alpha) s \rho_1 + s^2((b_2 -  a_2)\rho_2 - \tfrac 12(\tfrac 12 -  2\alpha) \rho_1^2 )) \\ & \quad + (1 - s\rho_1 + \tfrac 12 s^2 \rho_2)(- 1 + (1 - 2\alpha) t \rho_1) \\ & \quad 
    + (2 - 2(t+s-st)\rho_1 + \tfrac 12 (s^2 + 2st -2s^2t )(\tfrac 12 \rho_1^2 + (2 a_2 + 1)\rho_2))
    \\ & = 
    - 2 \alpha t \rho_1 - 2\alpha s \rho_1 \\ & \quad + s^2((b_2 - a_2) \rho_2 - \tfrac 12(\tfrac 12 - 2\alpha)\rho_1^2 - \tfrac 12 \rho_2 + \tfrac 12 (\tfrac 12 \rho_1^2 + (2 a_2 + 1)\rho_2))  \\ & \quad + st(-(1-2\alpha) \rho_1^2 - (1 - 2\alpha) \rho_1^2 + 2\rho_1 + (\tfrac 12 \rho_1^2 + (2 a_2 + 1)\rho_2)))  \\ & \quad + s^2t( -(b_2 - a_2)\rho_2 \rho_1 + \tfrac 12(\tfrac 12 + 2 \alpha)\rho_1^3 + \tfrac 12 (1 - 2\alpha) \rho_2 \rho_1 -  (\tfrac 12 \rho_1^2 + (2 a_2 + 1)\rho_2)) ) \\ & = 
    - 2\alpha t \rho_1 - 2\alpha s \rho_1 \\ & \quad + s^2( b_2 \rho_2 + \alpha \rho_1^2) \\ & \quad + st( (2a_2 + 1) \rho_2 + (4\alpha - \tfrac 32) \rho_1^2 + 2\rho_1)
    \\ & \quad + s^2t( (a_2 - b_2 +\tfrac 12 (1 - 2\alpha))\rho_2 \rho_1 + \tfrac 12(\tfrac 12 + 2\alpha)\rho_1^3 - (2 a_2 + 1)\rho_2- \tfrac 12 \rho_1^2 ).
\end{align*}
Comparing coefficients yields the result.    
\end{proof}

\begin{lemma}[Bounds on third moment]\label{l44}
    \sloppy Assume that $\sup_N \mathbf E[\rho_3(X_0^N)] < \infty$ and $a_2, a_3 < \tfrac 12$ (recall from \eqref{eq:an}). Then, for all $T>0$, there is $C_T < \infty$ with $\sup_N \sup_{0 \leq t \leq T} \mathbf E[\rho_3(X_t^N] < C_T$.
\end{lemma}

\begin{proof}
    Let us rearrange some results from Lemma~\ref{l41} and Lemma~\ref{l43}. We write
    \begin{align*}
        G_1  \left( 
        \begin{matrix}
            \rho_1 \\ \rho_1^2 \\ \rho_2 \\ \rho_1^3 \\ \rho_2 \rho_1 \\ \rho_3
        \end{matrix}
        \right) & = \underbrace{\left( 
        \begin{matrix}
            0 & 0 & 0 & 0 & 0 & 0 \\
            0 & 0 & 0 & 0 & 0 & 0 \\
            0 & \tfrac 14 & -\tfrac 12(1 - 2a_2) & 0 & 0 & 0 \\
            0 & 0 & 0 & 0 & 0 & 0 \\
            0 & 0 & 0 & \tfrac 14 & -\tfrac 12(1 - 2 a_2) & 0 \\ 0 & 0 & 0 & 0 & \tfrac{3}2 a_2 & - \tfrac 12(1 - 2a_3) 
            \end{matrix}
        \right)}_{=:M_1} \left( 
        \begin{matrix}
            \rho_1 \\ \rho_1^2 \\ \rho_2 \\ \rho_1^3 \\ \rho_2 \rho_1 \\ \rho_3
        \end{matrix}
        \right), \\
        G_0\underbrace{\left( 
        \begin{matrix}
            \rho_1 \\ \rho_1^2 \\ \rho_2 \\ \rho_1^3 \\ \rho_2 \rho_1 \\ \rho_3
        \end{matrix}
        \right)}_{:= \rho} & = 
        \underbrace{\left(
            \begin{matrix}
            \alpha & 0 & 0 & 0 & 0 & 0 \\
            1 & 2\alpha - \tfrac 34 & a_2 + \tfrac 12 & 0 & 0 & 0 \\
            0 & \alpha & b_2 & 0 & 0 & 0 \\
            0 & 3 & 0 & 3\alpha - \tfrac 94 & 3(a_2 +  \tfrac 12) & 0 \\
            0 & -\tfrac 12 & -(2a_2 + 1) & \tfrac 12(\tfrac 12  + 2\alpha) & a_2 - b_2 + \tfrac 12 (1 - 2\alpha) & 0 \\ 0 & 0 & 0 & 0 & 3\alpha a_2 + \tfrac 32 b_2 & b_3
            \end{matrix}
        \right)}_{=:M_0} \left( 
            \begin{matrix}
                \rho_1 \\ \rho_1^2 \\ \rho_2 \\ \rho_1^3 \\ \rho_2 \rho_1 \\ \rho_3
            \end{matrix}
            \right).
    \end{align*}
    An analysis using a computer algebra system yields: The matrix $NM_1 + M_0$ has eigenvalues $\lambda_1,...,\lambda_6$ with $\lambda_, \lambda_2, \lambda_3 = O(1)$, and $\lambda_4, \lambda_5 = - N(\tfrac 12 -a_2) + O(1), \lambda_6 = - N(\tfrac 12 -a_3) + O(1)$, 
    so since $a_2, a_3 < 1/2$, for every $T>0$, there is $c_T < \infty$ such that $\sup_{0 \le t \leq T} e^{\lambda_i t} \le c_T$. Moreover, we can represent the vector $e_6$ (in direction $\rho_3$) as a linear combinations of the corresponding eigenvectors $v_1,...,v_6$. So, $e_6 = \sum_{i=1}^6 a_i v_i$, where $v_i$ is eigenvector for the eigenvalue $\lambda_i$, $i=1,...,6$. Since $(e^{-\lambda_i t}v_i^\top \rho(X_t^N))_{t\geq 0}$ is a martingale, $i=1,...,6$, we can write
    \begin{align*}
        \sup_{0\leq t\leq T} \mathbf E[\rho_3(X_t^N)] & = \sup_{0\leq t\leq T} \sum_{i=1}^6 a_i  \mathbf E[v_i^\top\rho(X_t^N)] = \sup_{0\leq t\leq T} \sum_{i=1}^6 a_i e^{\lambda_i t} \mathbf E[v_i^\top\rho(X_0^N)] \\ & < C_T \sup_N \mathbf E[\rho_3(X_0^N)],
    \end{align*}
    for some $C_T < \infty$ only depending on $T$, since eigenvalues are either $O(1)$ or negative (use $a_2, a_3 < \tfrac 12$ here). This gives the result. 
\end{proof}

\begin{lemma}[A martingale]\label{l:tight}
    For each $N$, the process $(M_t^N)_{t\geq 0}$ with $M_t^N := e^{-\alpha t} \Phi(X_t^N)$ is a martingale with quadratic variation (recall $a_2$ from Lemma \ref{l41})
    \begin{align}
        \label{eq:mart2}
        &\Big(e^{-2\alpha t} \int_0^t F(X_s^N) ds\Big)_{t\geq 0}, \qquad 
        F(X_s^N) := (a_2 + \tfrac 12) \rho_2(X^N_k(s)) +  \rho_1(X^N_k(s)) - \tfrac 34 \rho_1^2(X^N_k(s)) \Big).
    \end{align}
\end{lemma}

\begin{proof}
Recall from \cite{EthierKurtz1986}, Lemma 4.3.2, that
\begin{align}\label{eq:mart1}
    e^{-\alpha t} & \Phi(X_t^N) + \int_0^t e^{-\alpha s} \big( \alpha \Phi(X_s^N) - G_0 \Phi (X_s^N)\big) ds = e^{-\alpha t} \Phi(X_t^N)
\end{align}
is a martingale. We compute its quadratic variation using Lemma~\ref{l43} by 
\begin{align*}
    [e^{-\alpha \cdot} & \Phi(X_\cdot^N)]_t = e^{-2\alpha t} [\Phi(X_\cdot^N)]_t] = e^{-2\alpha t} \int_0^t \Big( G_0 \Phi^2(X_s^N) - 2 \Phi (X_s^N) G_0 \Phi(X_s^N) \Big)ds.
\end{align*}
\end{proof}

\begin{proposition}\label{p:tight}
    Let $\mathcal D_E' := \{\psi_t : t \in [0,1]\}$, and assume that for all $T>0$, there is $C_T < \infty$ with $\sup_N \sup_{0 \leq t \leq T} \mathbf E[\rho_3(X_t^N] < C_T$.
    Then, $(\Phi(X_t^N)_{t\geq 0})_N$ is tight and $(X_t^N)_{t\geq 0}$ is tight in measure.
\end{proposition}

\begin{proof}
We will show the following: 
\begin{enumerate}
  \item (one-dimensional tightness) for every \(t\in[0,T]\) the family \((\Phi(X^N_t))_{N\ge1}\) is tight;
  \item (tightness of $(X^N)_{N\ge 1}$): the family \((X^N_t)_{N\ge1}\) is tight in measure; 
  \item (Aldous condition) for every \(\varepsilon>0\) and $T>0$, and every sequence of stopping times \(\tau_N\) bounded by \(T\), there exists a delta $\delta > 0$ such that 
  \[
  \lim_{\delta\downarrow 0}\limsup_{N\to\infty}\sup_{0\le\theta\le\delta}
  \mathbf P\big(|\Phi(X^N_{\tau_N+\theta})-\Phi(X^N_{\tau_N})|>\varepsilon\big)=0.
  \]
\end{enumerate}
Then, tightness of $(\Phi(X^N))_{N\ge1}$ in \(\mathcal D(\mathbb R_+)\) follows from 1.\ and 3.\ by the Aldous--Rebolledo criterion (see Theorem~1.17 in \cite{MR1779100}).  
The second claim, tightness in measure of $(X^N)_{N\ge1}$, equals 2.

~

We will be using the martingale $(M_t^N)_{t\geq 0}$ with $M_t^N := e^{-\alpha t} \Phi(X_t^N)$ and notation from Lemma~\ref{l:tight}, in particular for the quadratic variation of $(M_t^N)_{t\geq 0}$, as given in \eqref{eq:mart2}.

~

For 1., note that since $(e^{-\alpha t}\Phi(X_t^N))_{t\ge0}$ is a martingale -- see \eqref{eq:mart1} -- we have \(\mathbf E[\Phi(X_t^N)] = e^{\alpha t}\,\mathbf E[\Phi(X_0^N)]\) for every $t \geq 0$. Using the Markov inequality, for any $C>0$,
\[
\mathbf P\big(\Phi(X_t^N) > C\big)
\leq \frac{e^{\alpha t}\,\mathbf E[\Phi(X_0^N)]}{C}.
\]
\sloppy Since $\sup_N \mathbf E[\Phi(X_0^N)] < \infty$ under the assumptions of Theorem~\ref{T2}, we find that $\big(\Phi(X_t^N)\big)_{N\ge1}$ is tight for all $t\ge0$. 

For 2., we show that the sequence $(\Gamma_{X^N})_{N\ge 1}$, defined in \eqref{occmeas}, is tight in $\mathcal P([0,\infty)\times \mathcal P(\mathbb N_0))$. For each $N$, let $M_t^N$ be as in Lemma~\ref{l:tight}. Since this 
is a non-negative martingale with uniformly bounded initial expectations by assumption, $\sup_N \mathbf E[\Phi(X^N_0)] < \infty$. Fix $\varepsilon > 0$ and choose $T>0$ such that \(\int_0^T e^{-s}\, ds \ge 1-\varepsilon.\) Applying Doob's maximal inequality to the martingale $(M^N_t)_{t\geq 0}$, there exists a constant $\lambda_\varepsilon < \infty$ such that
\[
\mathbf P\Big( \sup_{0\le t \le T} M^N_t \le \lambda_\varepsilon \Big) \ge 1 - \varepsilon, \quad \text{for all } N.
\]
Observe that on this event
\[
\sup_{0\le t \le T} \Phi(X^N_t) = \sup_{0\le t \le T} e^{\alpha t} M^N_t \le (1 \vee e^{\alpha T}) \lambda_\varepsilon := C_\varepsilon.
\]
Define the relatively compact set
\[
K_{C_\varepsilon} := \{ x \in \mathcal P(\mathbb N_0) : \Phi(x) \le C_\varepsilon \}.
\]
Then, on the event $\{\sup_{0\le t \le T} \Phi(X^N_t) \le C_\varepsilon\}$, we have $X^N_t \in K_{C_\varepsilon}$ for all $t\in [0,T]$, which implies
\[
\Gamma_{X^N}([0,T] \times K_{C_\varepsilon}) = \int_0^T e^{-s} \mathbf 1_{\{X^N_s \in K_{C_\varepsilon}\}}\, ds \ge \int_0^T e^{-s}\, ds \ge 1-\varepsilon.
\]
Combining the probability bound with the inequality above, we obtain
\[
\mathbf P\Big( \Gamma_{X^N}([0,T] \times K_{C_\varepsilon}) \ge 1-\varepsilon \Big) \ge \mathbf P\Big( \sup_{0 \leq t \leq T} M_t^N \leq \lambda_\epsilon \Big) \ge 1-\varepsilon,
\]
which is precisely the tightness condition required by Prohorov's theorem. Consequently, the sequence $(X^N)_{N\ge 1}$ is tight in measure. 

For 3., recall the martingale $(M_t^N)_{t\geq 0}$ with $M_t^N := e^{-\alpha t} \Phi(X_t^N)$ from Lemma~\ref{l:tight} and its quadratic variation, as given in \eqref{eq:mart2}. Let $(\tau_N)_{N\ge 1}$ be stopping times bounded by $T>0$, and fix $\varepsilon>0$. By the assumption on finite third factorial moments, there exists a constant $C_T<\infty$ such that (recall $F$ from \eqref{eq:mart2})
\[
\sup_{N\ge 1}\sup_{0\le s\le T} \mathbf E[F(X^N_s)] \le C_T.
\]
Hence, for any $0\le \theta \le \delta$, we have
\begin{align*}
\mathbf E\Big[ & [M^N]_{\tau_N+\theta} - [M^N]_{\tau_N}\Big] 
= e^{-2\alpha \tau_N}\Big( e^{-2\alpha \theta} \mathbf E\Big[\int_0^{\tau_N+\theta} F(X^N_s) \,ds\Big] - \mathbf E\Big[\int_0^{\tau_N} F(X^N_s) ds\Big]\Big) \\
\\ & \leq (1 \vee e^{-2\alpha T}) \Big((e^{-2\alpha \theta} - 1) \mathbf E\Big[\int_0^T F(X^N_s) \,ds\Big] + \mathbf E\Big[\int_{\tau_N}^{\tau_N + \theta} F(X^N_s)\Big]\Big) \Big)
\\ & \leq (1 \vee e^{-2\alpha T})( 2\alpha \delta T C_T + \delta C_T) =: \delta C_T'.
\end{align*}
So, we may write
\begin{align*}
|\Phi(X^N_{\tau_N+\theta}) - \Phi(X^N_{\tau_N}) |
&= |e^{\alpha(\tau_N+\theta)} M^N_{\tau_N+\theta} - e^{\alpha \tau_N} M^N_{\tau_N} | \\
&= \Big|e^{\alpha \tau_N} \Big( (e^{\alpha \theta}-1) M^N_{\tau_N+\theta} + (M^N_{\tau_N+\theta} - M^N_{\tau_N}) \Big)\Big| \\ & \leq c_T \big( (e^{|\alpha| \delta} - 1) M_{\tau_N + \theta}^N + |M^N_{\tau_N+\theta} - M^N_{\tau_N}|\big)
\end{align*}
for $c_T := 1 \vee e^{\alpha T}$ and bound each term. By Doob's maximal inequality, 
\[
\mathbf P\Big( |M^N_{\tau_N+\theta} - M^N_{\tau_N}| > \frac \varepsilon {2c_T} \Big) 
\le \frac{4 c_T^2}{\varepsilon^2} \mathbf E\Big[ [M^N]_{\tau_N+\theta} - [M^N]_{\tau_N} \Big] 
\le \frac{4 c_T^2 \delta C'_T}{\varepsilon^2}.
\]
Choose \(\delta_1 := \frac{\varepsilon^3}{8 c_T^2 C_T'},\) so that this probability is at most $\varepsilon/2$, uniformly in $N$. For the multiplicative term, $|M^N_{\tau_N+\theta}| \le \sup_{0\le t\le T} |M^N_t|$ and we may choose \(K_\varepsilon>0\) such that
\(\sup_N \mathbf P\Big( \sup_{0\le t\le T} |M^N_t| > K_\varepsilon \Big) \le \varepsilon/2\). Choose $\delta_2 >0$ such that $(c_T e^{|\alpha| \delta_2}-1) K_\varepsilon \le \varepsilon/2$. Finally, let $\delta := \min\{\delta_1, \delta_2\}$. Then for all $0\le \theta \le \delta$,
\begin{align*}
\mathbf P\Big( |\Phi(X^N_{\tau_N+\theta}) & - \Phi(X^N_{\tau_N})| > \varepsilon \Big) \\ & \le 
\mathbf P\Big( c_T (e^{|\alpha| \delta}-1) |M^N_{\tau_N+\theta}| > \varepsilon/2 \Big) + \mathbf P\Big( c_T |M^N_{\tau_N+\theta} - M^N_{\tau_N}| > \varepsilon/2 \Big) \le \varepsilon,   
\end{align*}
uniformly in $N$. This verifies the Aldous condition for tightness.
\end{proof}

~

\begin{remark}\label{rem:outline}
Let us summarize what remains to be done for the proof of Theorem~\ref{T2}, as an application of Theorem~\ref{T1}:
\begin{enumerate}
    \item Show that the form of factorial moments as given in \eqref{eq:an} and \eqref{eq:bn} hold. 
    \item For the resulting $a_2$, we must have $a_2 < 1$, such that the assumptions of Lemma~\ref{l44} hold, which shows the required tightness of $(\Phi(X_t^N))_{t\geq 0}$ and $(X_t^N)_{t\geq 0}$ via Proposition~\ref{p:tight}; 
    \item For A3, show \eqref{eq:astast}, i.e.\ $G_1\psi_s(x) = 0$ implies that $x$ is Poisson (negative binomial);
    \item Show that the right hand side of \eqref{eq:astastast} corresponds to the generator of $Z$.
\end{enumerate}
We will prove these four assertions in the next section.     
\end{remark}

\section{Proof of Theorem~\ref{T2}}
\label{S:later}
\subsection{Case (i): Binomial/Poisson}
\label{ss:binom}
In this part, we focus on the case \eqref{eq:twocases}(i). The goal of this section is to prove 1., 2., and 3.\ from the end of section~\ref{S:preparation}. As explained there, this will conclude the proof of Theorem~\ref{T2}.1. Note that for $p_k = B(k, \tfrac 12)$,
\begin{align} \notag
    \psi_s(p_k) & = \sum_{i=0}^k \binom{k}{i} \frac{1}{2^k}(1-s)^i = \Big(1 - \frac s 2\Big)^k \\ \notag& = 1 - s \frac{k}{2} + \tfrac 12 s^2 \frac{k(k-1)}{4} - \tfrac 16 s^3 \frac{k(k-1)(k-2)}{8} + O(s^4), \text{ i.e.} \\ \notag
    \rho_1(p_k) & = \frac k 2, \qquad \rho_2(p_k) = \frac{k(k-1)}{4}, \qquad \rho_3(p_k) = \frac{k(k-1)(k-2)}{8}.
\end{align}
This shows \eqref{eq:an} with $a_2 = \tfrac 14$, $a_3 = \tfrac 18$. Next, for $p_k^N = B(k, \tfrac 12 + \varepsilon_N)$ with $N \varepsilon_N \xrightarrow{N\to\infty} \alpha \in \mathbb R$, we already computed $N \varepsilon_N$ recall that $\psi_s(B(k,p)) = (1-sp)^k$, so
\begin{align*}
    N(\psi_s(p_k^N) & - \psi_s(p_k) = N \Big(\big(1 - s\big(\tfrac 12 + \varepsilon_N\big)\big)^k - \big(1 - s\tfrac 12\big)^k\Big) \\ & \xrightarrow{N\to\infty} - \alpha ks (1 - s\tfrac 12)^{k-1} = - \alpha ks + \tfrac 14 k(k-1)s - \tfrac 18 k(k-1)(k-2)s^2 + O(s^4),
\end{align*}
which shows 1.\ and 2.\ from Remark~\ref{rem:outline} and gives \eqref{eq:bn} with $b_2 = \tfrac 14, b_3 = - \tfrac 18$.

Next, we turn to 3. Note that 
\begin{align}
    2G_1 \psi_s(x) & = \psi_{s/2}^2(x) - \psi_s(x). \label{eq:G1a} 
\end{align}
Recall that if $x = \rm{Poi}(\lambda),$ then $\lambda = \rho_1(x)$ and
$$ \psi_s(x) = e^{-\lambda} \sum_{k=0}^\infty \frac{\lambda^k}{k!}(1-s)^k = e^{-s\lambda}$$ 
and $s\mapsto \psi_s(x)$ characterizes $x$ uniquely. We immediately see that $G_1 \psi_s(x) = 0$ if $x$ is Poisson. The reverse implication is given next. Here is a version of Lemma 3.8 of \cite{pfaffelhuber2023diploid}.

\begin{lemma}[Characterization of Poisson distributions]\label{l:4}
  Let $\psi_s$ and $\rho_n$ be as in \eqref{eq:taylor}. Let $x \in \mathcal P(\mathbb N_0)$ with $\rho_1(x) < \infty$. Then the following are equivalent:
  \begin{enumerate}
    \item $x = \text{Poi}(\rho_1(x))$;
    \item For all $n=1,2,...$ and  $s_1,...,s_n\in[0,1]$, $$\psi_{s_1}(x) \cdots \psi_{s_n}(x) = \frac 1n \sum_{j=1}^n \psi_{s_j/2}^2(x) \prod_{\genfrac{}{}{0pt}{}{k=1}{k\neq j}}^n \psi_{s_{k}}(x).$$
\end{enumerate}
\end{lemma}

\begin{proof}
    Since generating functions uniquely determine probability distributions, 1.\ is equivalent to 1'. $\psi_s(x) = e^{-s\rho_1(x)}$. \\
    \noindent $1'.\Rightarrow 2.:$ By assumption we have    \begin{align*}
    \psi_{s_1}(x) \cdots \psi_{s_n}(x) = e^{-(s_1 + \cdots + s_n)\rho_1(x)}.
    \end{align*}
    Since the right hand side only depends on $s_1 + \cdots + s_n$, the result follows from summing indices in $$\psi_{s_1}(x) \cdots \psi_{s_n}(x) = \psi_{s_j/2}^2(x) \prod_{\genfrac{}{}{0pt}{}{k=1}{k\neq j}}^n \psi_{s_{k}}(x).$$ 
  
    \noindent $2.\Rightarrow 1'.:$ We start with the following observation: For $s>0$ let $(s_{kj})_{k\in \mathbb N, j=1,...,k}$ be asymptotically negligible (in the sense that\ $\sup_j |s_{kj}| \xrightarrow{k\to\infty} 0$) and $\sum_{j=1}^k s_{kj} = s$. Then, since $$\psi_{s_{kj}}(x) = \sum_{i=0}^\infty x_i (1-s_{kj})^i = 1 - (s_{kj} + o(s_{kj}))\sum_{i=0}^\infty ix_i$$ (where we have used that $\rho_1(x) < \infty$), we have
    \begin{align}\label{eq:911}
        \log\Big(\prod_{j=1}^k \psi_{s_{kj}}(x)\Big) & = \sum_{j=1}^k \log(1 - (s_{kj} + o(s_{kj}))\rho_1(x)) \xrightarrow{k\to\infty}  -s\rho_1(x).
    \end{align}
    \sloppy Now, we come to proving the assertion: Fix $s\in [0,1]$, and let $\mathcal P_n$ be a random partition of $[0,s)$ with $n$ elements, which arises iteratively as follows: Starting with $\mathcal P_1 = \{[0,s)\}$, let $\mathcal P_{n+1}$ arise from $\mathcal P_n$ by randomly taking one partition element $[a,b)$ from $\mathcal P_n$, and adding the two elements $[a,(a+b)/2)$ and $[(a+b)/2, b)$ to $\mathcal P_{n+1}$. (We can e.g.\ have $\mathcal P_1 = \{[0,s)\}, \mathcal P_2 = \{[0,s/2), [s/2,s)\}, \mathcal P_3 = \{[0,s/4), [s/4, s/2), [s/2, s)\}, \mathcal P_4 = \{[0,s/4), [s/4, 3s/8), [3s/8, s/2), [s/2, s)\},...)$. From 2., we  find iteratively, almost surely
    \begin{align*}
        \psi_s(x) & = \prod_{\pi \in \Pi_n} \psi_{|\pi|}(x) \mathbf P(\mathcal P_n=\Pi_n).
    \end{align*}
    It is not hard to see that -- almost surely -- every partition element in $\mathcal P_n$ eventually gets split in two, so $\{|\pi|: \pi \in \mathcal P_n\}$ is asymptotically negligible as $n\to\infty$. Therefore, 
    $$\prod_{\pi \in \Pi_n} \psi_{|\pi|}(x) \mathbf P(\mathcal P_n = \Pi_n) \xrightarrow{n\to\infty} e^{-s\rho_1(x)}$$ almost surely by \eqref{eq:911} and dominated convergence. Combining the last two equalities gives 1'. 
\end{proof}

\begin{remark}\label{r:4}
Note that from \eqref{eq:G1a}, since $G_1$ is a first derivative, 
\begin{align*}
    2G_1 (\psi_{s_1} \cdots \psi_{s_n})(x) & = \Big(\frac 1n \sum_{j=1}^n \psi_{s_j/2}^2(x) \prod_{\genfrac{}{}{0pt}{}{k=1}{k\neq j}}^n \psi_{s_{k}}(x)\Big) - \psi_{s_1}(x) \cdots \psi_{s_n}(x).
\end{align*}
In particular, Lemma~\ref{l:4} shows that for some $x$ with $\rho_1(x) < \infty$, using again that $G_1$ is a first derivative, we have $G_1 \psi_s(x) = 0$ for all $s \in [0,1]$ iff $x = \text{Poi}(\rho_1(x))$. 
\end{remark}

\begin{lemma}\label{l:astastbinom}
    If $\Xi(z) = \text{Poi}(z)$, then 
    $$ G_0(g \circ \Phi)\circ \Xi(z) = \alpha z g'(z) + \tfrac 12 z g''(z).$$    
\end{lemma}

\begin{proof}
    This is straight-forward from \eqref{eq:astastast}, since the variance of a Poisson distribution coincides with its parameter.
\end{proof}

\subsubsection*{Proof of Theorem~\ref{T2}.1}
\label{ss:proof311}
As announced at the end of Section~\ref{S:model}, we have to show \eqref{eq:astast} for A3, which is the precise result from Remark~\ref{r:4} (based on Lemma~\ref{l:4}). From~\ref{p:tight}, we see -- based on finite third moments --  that $((\Phi(X_t^N))_{t\geq 0})_N$ is tight and $(X_t^N)_{t\geq 0}$ is tight in measure. In particular, A1 holds along subsequences. Last, the form of the generator of $Z$ is given in Lemma~\ref{eq:astastast}. Noting that $v(\Xi(z)) = z$ (the variance of a Poisson distribution coincides with its parameter) we are done. 

\subsection{Case (ii): uniform/negative binomial}
\label{ss:uniform}
In this part, we focus on the case \eqref{eq:twocases}(ii). The goal of this section is to prove 1.--4.\ from Remark~\ref{rem:outline}. Note that
\begin{align*}
    \psi_s(p_k) & = \sum_{j=0}^k \frac{1}{k+1} (1-s)^j = \frac{1}{k+1} \frac{1 - (1-s)^{k+1}}{s} = \frac 1s \int_0^s (1-r)^k dr \\ & = \frac{1}{k+1}\sum_{j=0}^{k} \binom{k+1}{j+1} (-s)^{j} = 1 - s \frac{k}{2} + \tfrac 12 s^2 \frac{k(k-1)}{3} - \tfrac 16 s^3 \frac{k(k-1)(k-2)}{4} + O(s^4), \text{ i.e.},\\
    \rho_1(p_k) & = \frac{k}{2}, \qquad \rho_2(p_k) = \frac{k(k-1)}{3}, \qquad \rho_3(p_k) = \frac{k(k-1)(k-2)}{4}, 
\end{align*}
which shows \eqref{eq:an} with $a_2 = \tfrac 13 < 1$ and $a_3 = \tfrac 14$. Moreover, \eqref{eq:bn} holds since $r_k=0$ for all $k$. This shows 1.\ and 2. For 3., from Lemma~\ref{l41},
\begin{align}\label{eq:G1b}
    2G \psi_s(x) & = \Big( \frac 1s \int_0^t \psi_r(x) dr \Big)^2 - \psi_s(x).    
\end{align}

We denote by NB$(k,p)$ the negative binomial distribution, i.e.\ the distribution of the number of failures in a Bernoulli experiment with success probability $p$ until the $k$th success. Recall that expectation and variance are given by $\rho_1(\text{NB}(2, p)) = \frac{2(1-p)}{p}$ and $v(\text{NB}(2, p)) = \frac{2(1-p)}{p^2}$. In other words, for $p = \frac{2}{z+2}$, we have 
\begin{align}\label{eq:expvarNB}
    \rho_1\Big(\text{NB}\Big(2, \frac{2}{z+2}\Big)\Big) = z, \qquad \rho_2\Big(\text{NB}\Big(2, \frac{2}{z+2}\Big)\Big) = \tfrac 32 z^2, \qquad v\Big(\text{NB}\Big(2, \frac{2}{z+2}\Big)\Big) = \tfrac 12 z(z+2).
\end{align}

\begin{lemma}[Characterization of a negative binomial distribution]
\label{l:5}
    Let $x \in \mathcal P(\mathbb N)$ with $\rho_1(x) < \infty$. Then, the following are equivalent:
    \begin{enumerate}
        \item $x = \text{NB}(2, p)$;
        \item  $p = \frac{2}{z + 2}$ with $\rho_1(x) = \psi_0'(x) = z$ and for all $t\in[0,1]$, we have 
        $$\Big(\frac{1}{t}\int_0^t \psi_s(x) ds\Big)^2 - \psi_t(x) = 0.$$ 
    \end{enumerate}
\end{lemma}

\begin{proof}
    Recall that for $x = \text{NB}(k,p)$ and $z = \frac 2p - 2$ (which is equivalent to $p = \frac{2}{z+2}$),
    $$ \psi_t(x) = \Big(\frac{p}{1 - (1 - p)(1-t)}\Big)^k = \Big(\frac{2}{2 + t z}\Big)^k.$$
    From this, 1.$\Rightarrow$2.\ is a straight-forward calculation. For 2.$\Rightarrow$1., we study the integral equation 
    \begin{align}
        \label{eq:uni1}
        \Big(\frac{1}{t}\int_0^t \psi_s(x) ds\Big)^2 - \psi_t(x) = 0.
    \end{align}
    Define 
    $$ \beta_s(x) := \frac{1}{s}\int_0^s \psi_r(x) ds,$$
    and note that \eqref{eq:uni1} for all $t$ implies, by integration 
    \begin{align*}
        0 = \int_0^t \beta_s^2(x) ds - t \beta_t(x),
    \end{align*}
    i.e.\ (by taking derivatives wrt $t$)
    \begin{align} \label{eq:initial}
        \beta_t^2(x) - \beta_t(x) - t \frac{d}{dt} \beta_t(x) = 0.
    \end{align}
    Note that, taking another derivative wrt t at $t=0$,
    \begin{align*}
        \frac{d}{dt} \beta_t^2 - \beta_t - t\frac{d}{dt} \beta_t(x) \Big|_{t=0} = \frac{d}{dt} \beta_t^2 - 2\beta_t = 0,
    \end{align*}
    i.e.\ in order to have a unique solution of the initial value problem \eqref{eq:initial}, we need to fix $z := 2\beta'_0(x)$. In addition,
    $$ \beta_0'(x) = \lim_{t\to 0} \frac{\frac 1t \int_0^t \psi_r(x) dr - 1}{t} = \lim_{t\to 0} \frac{\frac 1t \int_0^t r \psi'_0(x)  + o(r) dr}{t} = \frac 12 \psi'_0(x).$$
    Since the ODE \eqref{eq:initial} satisfies the usual Lipschitz condition, it has a unique solution with $2\beta'_0(x) = z := \rho_1(x)$, which can be computed using separation of variables, and is given by
    $$ \beta_t(x) = \frac{1}{1 + tz/2}, \qquad \text{so} \qquad \psi_t(x) = \beta_t^2(x) = \Big(\frac{2}{2 + tz}\Big)^2.$$
    The claim follows since $t\mapsto \psi_t(x)$ determines $x$ uniquely.
\end{proof}

\begin{remark}\label{r:5}
Note that from \eqref{eq:G1b}, 
\begin{align*}
    2G_1 \psi_{t} (x) & = \Big(\frac 1t 
 \int_0^t \psi_s(x) ds \Big)^2 - \psi_t(x) .
\end{align*}
In particular, Lemma~\ref{l:5} shows that for some $x$ with $\rho_1(x) < \infty$, we have $G_1 f(x) = 0$ for all $f \in \mathcal D_E' := \big\{\psi_{t}: t \in [0,1]\big\}$
iff $x = \text{NB}(2, 2/(\rho_1(x) + 2))$. 
\end{remark}

\begin{lemma}
    If $\Xi(z) = \text{NB}(2, \tfrac{2}{2+z})$, then $\rho_1(\Xi(z)) = z$ and 
    $$ G_0(g \circ \Phi) \circ \Xi(z) = \tfrac 12 z(z+2)g''(\Phi(z)).$$
\end{lemma}

\begin{proof}
    This is straight-forward from \eqref{eq:astastast}, since the variance of NB$\Big(2, \frac{1}{z+1}\Big)$ is $2z(z+1)$; see \eqref{eq:expvarNB}.
\end{proof}

\subsection*{Proof of Theorem~\ref{T2}.2}
\label{ss:proof312}
We proceed as in the proof of Theorem~\ref{T2}.1. Again, for A3, see Remark~\ref{r:5} (based on Lemma~\ref{l:5}). Again, A1 holds along a subsequence. Last, for the form of the generator of $Z$ as given in Lemma~\ref{eq:astastast}, note that $v(\Xi(z)) = \tfrac 12 z(z+2)$ (see \eqref{eq:expvarNB}), so we are done. 

\subsection*{Acknowledgements}
\noindent
We thank Emmanuel Schertzer for bringing \cite{otto2023structured} to our attention. PP is supported by the Freiburg Center for Data Analysis, Modeling, and AI.


\end{document}